\documentclass[a4paper,12pt]{amsart}

\usepackage{amsfonts}
\usepackage{amsmath}
\usepackage{amssymb}

\usepackage{mathrsfs}
\usepackage{hyperref}
\usepackage{graphicx}
\theoremstyle{plain}
 \newtheorem{theorem}{Theorem}[section]
 
 \newtheorem{lemma}{Lemma}[section]
 \newtheorem{property}{Property}[section]
 \newtheorem{corollary}{Corollary}[section]
\theoremstyle{definition}
 \newtheorem{example}{Example}[section]
 \newtheorem{definition}{Definition}[section]
\theoremstyle{remark}
 \newtheorem{remark}{Remark}[section]
 \numberwithin{equation}{section}

\title[Lyapunov and Hartman-Wintner type inequalities]{Lyapunov and Hartman-Wintner type inequalities for a nonlinear fractional boundary value problem with generalized Hilfer derivative}



\author[Kirane]{\bfseries Mokhtar Kirane$^{1, 2}$}
\address{$^{1}$LaSIE, Facult\'{e} des Sciences, Pole Sciences et Technologies, Universit\'{e} de La Rochelle, Avenue M. Crepeau, 17042 La Rochelle Cedex, France}
\address{$^{2}$Nonlinear Analysis and Applied Mathematics (NAAM) Research Group, Department of Mathematics, Faculty of Science, King Abdulaziz University, P.O. Box 80203, Jeddah 21589, Saudi Arabia}
\email{mkirane@univ-lr.fr}

\author[Torebek]{\bfseries Berikbol T. Torebek$^{3}$}

\address{$^{3}$Department of Differential Equations,\\ Institute of Mathematics and Mathematical Modeling\\
Pushkin street 125, 050010, Almaty, Kazakhstan}
\email{torebek@math.kz}

\begin{document}

\vspace{18mm} \setcounter{page}{1} \thispagestyle{empty}

\begin{abstract}
In this work, we obtain a Lyapunov-type and a Hartman-Wintner-type inequalities for a linear and a nonlinear fractional differential equation with generalized Hilfer operator subject to Dirichlet-type boundary conditions. We prove existence of positive solutions to a nonlinear fractional boundary value problem. As an application, we obtain a lower bound for the eigenvalues of corresponding equations.
\end{abstract}

\maketitle

\section{Introduction and main results}

Lyapunov's inequality is an outstanding result in mathematics with many applications -- see \cite{Tiryaki10, Hashizume15, SunLiu15} and references therein. The result, as proved by Lyapunov in \cite{Lyapunov93}, asserts that if $q\in C\left([a,b];\mathbb{R}\right),$ then a necessary condition for the boundary value problem \begin{equation}\label{1}\left\{\begin{array}{l}u''(t)+q(t)u(t)=0,\,a<t<b,\\ {}\\u(a)=u(b)=0,\end{array}\right.
\end{equation} to have a nontrivial classical solution is given by
\begin{equation}\label{2}\int\limits_{a}^{b}\left|q(s)\right|ds>\frac{4}{b-a}.\end{equation}

Looking for a generalization for fractional differential equations, in \cite{Ferreira13}, Ferreira investigated a Lyapunov-type inequality for the Riemann-Liouville fractional boundary value problem
\begin{equation}\label{3} \left\{\begin{array}{l}D^{\alpha}_{a}u(t)+q(t)u(t)=0,\,a<t<b,\\{} \\u(a)=u(b)=0,\end{array}\right.\end{equation}
where $D^{\alpha}_{a}$ is the (left) Riemann-Liouville derivative of order $\alpha\in(1,2]$ and $q\in C\left([a,b];\mathbb{R}\right).$ He proved that, if \eqref{3} has a nontrivial classical solution, then \begin{equation}\label{4}\int\limits_a^b\left|q(s)\right|ds> \Gamma(\alpha)\left(\frac{4}{b-a}\right)^{\alpha-1}.\end{equation} In \cite{Ferreira14} Ferreira considered the fractional boundary value problem with $\mathcal{D}^{\alpha}_a$ Caputo derivative of order $1<\alpha\leq 2$ \begin{equation}\label{5}\left\{\begin{array}{l}\mathcal{D}^{\alpha}_a u(t)+q(t)u(t)=0,\,a<t<b,\\{}\\u(a)=u(b)=0,\end{array}\right.\end{equation} and an interesting Lyapunov-type inequality was estabilished
\begin{equation}\label{6}\int\limits_a^b \left|q(s)\right|ds>\frac{\Gamma(\alpha)\alpha^\alpha}{\left((\alpha-1)(b-a)\right)^{\alpha-1}},\end{equation} where $q\in C\left([a,b];\mathbb{R}\right).$

Moreover, some Lyapunov-type inequalities for linear fractional boundary value problems have been obtained in \cite{MaMaWang16, JieliSamet15, O'ReganSamet15, RongBai15, EshaghiAnsari16, Ferreira16, ChidouhTorres17}.

Recently, Chidouh and Torres \cite{ChidouhTorres17} obtained the following Lyapunov-type inequality for the Riemann-Liouville fractional-order nonlinear boundary value problem \begin{equation}\label{6*}\left\{\begin{array}{l}D^{\alpha}_{a}u(t)+q(t)f(u)=0, \,a<t<b,\\{}\\u(a)=u(b)=0,\end{array}\right.\end{equation} where $q:\, [a,b]\rightarrow \mathbb{R}$ is a real nontrivial Lebesgue integrable function and $f\in C\left(\mathbb{R}_+, \mathbb{R}_+\right)$ is a concave and nondecreasing function. They demonstrated that, if \eqref{6*} has a nontrivial solution, then \begin{equation}\label{6**}\int\limits_a^b\left|q(s)\right|ds> \frac{\Gamma(\alpha)4^{\alpha-1}\eta}{\left(b-a\right)^{\alpha-1}f(\eta)},\end{equation} where $\eta=\max\limits_{a\leq t\leq b}|u|.$

In \cite{HartmanWintner}, Hartman and Wintner proved that if \eqref{1} have a nontrivial solution, then \begin{equation}\label{HW}\int\limits^b_a (b-s)(s-a)q^+(s)ds >b-a,\end{equation} where $q^+(s)=\max\{q(s),0\}.$ Observe that Lyapunov inequality \eqref{2} can be deduced from \eqref{HW} using the fact that $$\max\limits_{a\leq s\leq b}(b-s)(s-a)=\frac{(b-a)^2}{4}.$$

Recently, some Hartman-Wintner-type inequalities were obtained for different fractional boundary value problems. In this direction, we refer to Cabrera, Sadarangania, Samet \cite{CabreraSadaranganoSamet} and Jleli, Kirane, Samet \cite{JleliKiraneSamet}.

In this paper we succeeded to generalize inequalities \eqref{2}, \eqref{4}, \eqref{6} and \eqref{6**} for the nonlinear fractional boundary value problem
\begin{equation}\label{7} \left\{\begin{array}{l}D^{\alpha, \mu}_{a}u(t)+q(t)f(u)=0, \,a<t<b,\\{}\\u(a)=u(b)=0,\end{array}\right.\end{equation} where $D^{\alpha, \gamma}_a,\, 1<\alpha\leq\gamma<2$ is a generalized Hilfer fractional derivative (see Section \ref{Sec2}).

Recently, nonlinear fractional differential equations have been investigated extensively. The motivation for those works rises from both the development of the theory of fractional calculus itself and the applications of such constructions in various sciences such as physics, chemistry, aerodynamics, electrodynamics of complex medium, and so on. For examples and details, see \cite{BabakhaniGejji, Zhang03, Kaufmann, Bonillaetal, DaftardarBhalekar, CabadaWang, CabadaInfante, InfanteRihani} and the references therein. The Dirichlet problem for fractional differential equations and systems of fractional differential equations is discussed in papers \cite{BaiLu, ShiZhang, JiangYuan, ZhangLiu}. In \cite{BaiLu}, the authors investigate the existence and multiplicity of positive solutions of problem \eqref{7}, for $q(t)\equiv 1,$ $\gamma=\alpha,$ $a=0,\,b=1.$

The paper is organized as follows. In Section \ref{Sec2} we recall some properties and definitions of fractional operators. In Section \ref{Sec3}, we firstly derive the corresponding Green's function; consequently problem \eqref{7} is reduced to a equivalent Fredholm integral equation of the second kind. Finally, we prove Lyapunov-type inequality for the linear case of problem \eqref{7}. In Section \ref{Sec4}: using the Krasnoselskii fixed point theorem, the existence and multiplicity of positive solutions are obtained. In Section \ref{Sec5}, assuming that $f: \mathbb{R}_+ \rightarrow \mathbb{R}_+$ is continuous, concave and nondecreasing, we generalize Lyapunov's inequality \eqref{6**}. In Section \ref{Sec6}, for continuous, concave and nondecreasing $f: \mathbb{R}_+ \rightarrow \mathbb{R}_+$ we obtain Hartman-Wintner's type inequality \eqref{17}.

\section{Definitions and some properties of fractional operators}\label{Sec2}

In this section, we compile some basic definitions and properties of fractional differential operators.

\begin{definition} \cite{KilbasSrivastavaTrujillo} (Riemann-Liouville integral). Let $f$ be a locally integrable real-valued function on $-\infty\leq a<t<b\leq+\infty.$ The Riemann--Liouville fractional integral $I_{a} ^\alpha$ of order $\alpha\in\mathbb R$ ($\alpha>0$) is defined as
$$
I_{a} ^\alpha  f\left( t \right) = \left(f*K_{\alpha}\right)(t) ={\rm{
}}\frac{1}{{\Gamma \left( \alpha \right)}}\int\limits_a^t {\left(
{t - s} \right)^{\alpha  - 1} f\left( s \right)} ds,$$
where $K_{\alpha}(t)=\frac{t^{\alpha-1}}{\Gamma(\alpha)},$ $\Gamma$ denotes the Euler gamma function.
\end{definition}

\begin{definition} \cite{KilbasSrivastavaTrujillo} (Riemann-Liouville derivative). Let $f\in L^1[a,b],$ $-\infty\leq a<t<b\leq+\infty$ and $f*K_{m-\alpha}\in W^{m,1}[a,b], m=[\alpha]+1, \alpha>0,$ where $W^{m,1}[a,b]$ is the Sobolev space defined as $$W^{m,1}[a,b]=\left\{f\in L^1[a,b]:\,\frac{d^m}{dt^m}f\in L^1[a,b]\right\}.$$ The Riemann--Liouville fractional derivative $D_{a} ^\alpha$ of order $\alpha>0$ ($m-1<\alpha<m,\,m\in \mathbb{N}$) is defined as
$$D_{a} ^\alpha f \left( t \right) = \frac{{d^m }}{{dt^m }}I_{a} ^{1 - \alpha } f \left( t \right)={\rm{}}\frac{1}{{\Gamma \left( m-\alpha \right)}}\frac{d^m}{dt^m}\int\limits_a^t {\left({t - s} \right)^{m-1-\alpha} f\left( s \right)} ds.$$
\end{definition}

\begin{definition} \cite{KilbasSrivastavaTrujillo} (Caputo derivative). Let $f\in L^1[a,b],$ $-\infty\leq a<t<b\leq+\infty$ and $f*K_{m-\alpha}\in W^{m,1}[a,b], m=[\alpha], \alpha>0.$ The Caputo fractional derivative $\mathcal{D}_a^\alpha$ of order $\alpha\in\mathbb R$ ($m-1<\alpha<m,\,m\in \mathbb{N}$) is defined as
\begin{align*}\mathcal{D}_{a} ^\alpha  f \left( t \right) \\&= D_{a}
^\alpha  \left[ f\left( t \right) - f\left( a \right) -f'\left( a \right)\frac{(t-a)}{1!}-... - f^{(m-1)}\left( a \right)\frac{(t-a)^{m-1}}{(m-1)!}\right].\end{align*}

If $f\in C^m[a,b]$ then, the Caputo fractional derivative $\mathcal{D}_a^\alpha$ of order $\alpha\in\mathbb R$ ($m-1<\alpha<m,\,m\in \mathbb{N}$) is defined as $$\mathcal{D}_{a}^\alpha  \left[ f \right]\left( t \right) = I_{a} ^{1 - \alpha } f^{(m)}\left( t \right)={\rm{}}\frac{1}{{\Gamma \left( m-\alpha \right)}}\int\limits_a^t {\left({t - s} \right)^{m-1-\alpha} f^{(m)}\left( s \right)} ds.$$
\end{definition}

\begin{definition} \cite{Hilfer} (Hilfer derivative). Let $f\in L^1[a,b],$ $-\infty\leq a<t<b\leq+\infty,$ $\mu\in[0,1],$ $m-1<\alpha<m,$ $f*K_{(1-\mu)(m-\alpha)}\in AC^m[a,b],$ where $AC^m[a,b]$ is  the space of real-valued functions $f$ which have continuous derivatives up to order $m-1$ on $[a,b]$ such that $f^{(m-1)}$ belongs to the space of absolutely continuous functions $AC[a,b]:$
$$AC^m[a,b]=\left\{f: [a,b]\rightarrow \mathbb{R}:\, f^{(m-1)}\in AC[a,b]\right\}.$$ The Hilfer fractional derivative $D_a^{\alpha,\mu}$ of order $\alpha\in\mathbb R$ ($m-1<\alpha<m,\,m\in \mathbb{N}$) and type $\mu$ is defined as $$D_a^{\alpha,\mu}f(t)=I_a^{\mu(m-\alpha)}\frac{d^m}{dt^m} I_a^{(1-\mu)(m-\alpha)}f(t).$$
\end{definition}

The Hilfer fractional derivative $D_a^{\alpha,\mu}$  is considered as an interpolator between the Riemann-Liouville and Caputo derivative since

$$D_a^{\alpha,0}=D_a^{\alpha},\,\,\,D_a^{\alpha,1}=\mathcal{D}_a^{\alpha}.$$

\begin{definition}\label{def5} \cite{FuratiIyiolaKirane} (generalized Hilfer derivative). Let $f\in L^1[a,b],$ $-\infty\leq a<t<b\leq+\infty,$ $\alpha\leq\mu,$ $m-1<\alpha<m,$ $f*K_{m-\gamma}\in AC^m[a,b],$ where $AC^m[a,b]$ is  the space of real-valued functions $f$ which have continuous derivatives up to order $m-1$ on $[a,b]$ such that $f^{(m-1)}$ belongs to the space of absolutely continuous functions $AC[a,b]:$
$$AC^m[a,b]=\left\{f: [a,b]\rightarrow \mathbb{R}:\, f^{(m-1)}\in AC[a,b]\right\}.$$ The generakized Hilfer fractional derivative $D_a^{\alpha,\gamma}$ of order $\alpha\in\mathbb R$ ($m-1<\alpha<m,\,m\in \mathbb{N}$) and type $\gamma$ is defined as $$D_a^{\alpha,\gamma}f(t)=I_a^{\gamma-\alpha}\frac{d^m}{dt^m} I_a^{m-\gamma}f(t).$$
\end{definition}

The generalized Hilfer fractional derivative $D_a^{\alpha,\gamma}$  is considered as an interpolator between the Riemann-Liouville, Caputo and Hilfer derivative since

$$D_a^{\alpha,\gamma}=\left\{ \begin{array}{l}D_a^{\alpha,\alpha}=D_a^{\alpha},\\{}\\ D_a^{\alpha,m}=\mathcal{D}_a^{\alpha},\\{}\\ D_a^{\alpha,\mu(m-\alpha)+\alpha}=D_a^{\alpha,\mu}.\end{array}\right.$$

\begin{property}\label{pr1} If $f\in L^1[a,b]$ and $\alpha>0,\, \beta>0,$ then the following equality holds $$I_a^\alpha I_a^\beta f(t)=I_a^{\alpha+\beta}f(t).$$
\end{property}

\begin{property}\label{pr2} \cite{KilbasSrivastavaTrujillo} If $f\in L^1[a,b]$ and $f*K_{m-\alpha}\in AC^m[a,b],$ $m=[\alpha]+1,$ then the equality $$I_a^\alpha D_a^\alpha f(t)=f(t)- \sum_{j=1}^m\frac{(t-a)^{\alpha-j}}{\Gamma(\alpha-j+1)}\frac{d^{m-j}}{dt^{m-j}}I_a^{m-\alpha}f(a),$$ holds almost everywhere on $[a,b].$\end{property}

\begin{property}\label{pr3} Let $f\in L^1[a,b]$ and $f*K_{m-\gamma}\in AC^m[a,b],$ $m-1<\alpha\leq\gamma<m,$ $m=[\alpha]+1.$ Then the Riemann-Liouville fractional integral $I_a^{\alpha}$ and the generalized Hilfer derivative $D_a^{\alpha,\gamma}$ are connected by the relation $$I_a^\alpha D_a^{\alpha,\gamma} f(t)=f(t)- \sum_{j=1}^m\frac{(t-a)^{\gamma-j}}{\Gamma(\gamma-j+1)}\frac{d^{m-j}}{dt^{m-j}}I_a^{m-\gamma}f(a).$$ \end{property}

\begin{proof} Using the representation $$D_a^{\alpha,\gamma} f(t)=I_a^{\gamma-\alpha}D_a^{\gamma} f(t)$$ and applying property \ref{pr1} we get
$$I_a^\alpha D_a^{\alpha,\gamma} f(t)=I_a^\alpha I_a^{\gamma-\alpha}D_a^{\gamma} f(t)=I_a^{\gamma}D_a^{\gamma} f(t).$$ Further, by property \ref{pr2} we have
\begin{align*}I_a^\alpha D_a^{\alpha,\gamma} f(t)&=I_a^{\gamma}D_a^{\gamma} f(t) \\&=f(t)- \sum_{j=1}^m\frac{(t-a)^{\gamma-j}} {\Gamma(\gamma-j+1)}\frac{d^{m-j}}{dt^{m-j}}I_a^{m-\gamma}f(a).\end{align*}
\end{proof}

\begin{property}\label{pr4} The following result holds true for the fractional derivative operator $D_a^{\alpha,\gamma}$ defined by definition \ref{def5}: $$D_a^{\alpha,\gamma}(t-a)^{\nu-1}= \frac{\Gamma(\nu)}{\Gamma(\nu-\alpha)}(t-a)^{\nu-\alpha-1},\,\,\,0<\nu,\,0<\alpha\leq\gamma.$$
\end{property}

\section{Lyapunov-type inequality for the linear case of problem \eqref{7}}\label{Sec3}

\begin{lemma}\label{lm1}The function $u(t)$ is a solution of the boundary value problem \eqref{7} if, and only if, $u(t)$ satisfies the integral equation \begin{equation}\label{*}u(t)=\int\limits_a^b G(t,s)q(s)f\left(u(s)\right)ds,\end{equation} where \begin{equation}\label{Green}G(t,s)=\left\{\begin{array}{l}\left(\frac{t-a}{b-a}\right)^{\gamma-1}\frac{(b-s)^{\alpha-1}} {\Gamma(\alpha)} -\frac{(t-s)^{\alpha-1}} {\Gamma(\alpha)},\,\,a\leq s\leq t\leq b,\\{}\\ \left(\frac{t-a}{b-a}\right)^{\gamma-1}\frac{(b-s)^{\alpha-1}} {\Gamma(\alpha)},\,\,a\leq t\leq s\leq b.\end{array}\right.\end{equation}\end{lemma}
\begin{proof}
Taking Riemann-Liouville fractional integral $I_a^{\alpha}$ to both side of the equation \begin{equation}\label{7*}D_a^{\alpha,\gamma}u(t)+q(t)f\left(u(s)\right)=0,\,a<t<b,\end{equation} and using property \ref{pr3}, the solution of the fractional differential equation \eqref{7*} can be written as \begin{align*}u(t)=C_1 \frac{(t-a)^{\gamma-2}} {\Gamma(\gamma-1)} &+C_2 \frac{(t-a)^{\gamma-1}} {\Gamma(\gamma)}\\&- {\rm{
}}\frac{1}{{\Gamma \left( \alpha \right)}}\int\limits_a^t {\left(
{t - s} \right)^{\alpha  - 1}q\left( s \right) f\left(u(s)\right)} ds,\, a<t<b,\end{align*} where $C_1$ and $C_2$ are real constants given by $$C_1=I_a^{2-\gamma}u(a),$$ $$C_2=\frac{d}{dt}I_a^{2-\gamma}u(a).$$ Since $u(a)=0,$ we get $C_1=0.$ From $u(b)=0,$ we have
$$C_2=\frac{\Gamma(\gamma)}{(b-a)^{\gamma-1}}I_a^\alpha q(b)f\left(u(b)\right).$$
Then for the function $u(t),$ we get
\begin{align*}u(t)=C_2 \frac{(t-a)^{\gamma-1}} {\Gamma(\gamma)}&-I_a^{\alpha}q(t)f\left(u(t)\right)\\& =\left(\frac{t-a}{b-a}\right)^{\gamma-1}I_a^{\alpha}q(b)f\left(u(b)\right)- I_a^{\alpha}q(t)f\left(u(t)\right)\\&=\left(\frac{t-a}{b-a}\right)^{\gamma-1} {\rm{
}}\frac{1}{{\Gamma \left( \alpha \right)}}\int\limits_a^b {\left({b - s} \right)^{\alpha  - 1}q\left( s \right) f\left(u(s)\right)} ds\\&-{\rm{
}}\frac{1}{{\Gamma \left( \alpha \right)}}\int\limits_a^t {\left({t - s} \right)^{\alpha  - 1}q\left( s \right) f\left(u(s)\right)} ds,\, a<t<b.\end{align*}
Whereupon \begin{align*}u(t)&=\frac{1}{{\Gamma \left( \alpha \right)}}\int\limits_a^t {\left(\left(\frac{t-a}{b-a}\right)^{\gamma-1}\left({b - s} \right)^{\alpha  - 1}-\left({t - s} \right)^{\alpha  - 1}\right)q\left( s \right) f\left(u(s)\right)} ds\\&+\frac{1}{{\Gamma \left( \alpha \right)}}\int\limits_t^b \left(\frac{t-a}{b-a}\right)^{\gamma-1}{\left({b - s} \right)^{\alpha  - 1}q\left( s \right) f\left(u(s)\right)} ds,\, a<t<b.\end{align*} This ends the proof.
\end{proof}
\begin{lemma}\label{lm2} Let Green's function $G(t,s)$ be defined as in Lemma \ref{lm1}. Then the Green function $G$ satisfies the following conditions:

\vspace{2mm}

(i) $G(t,s)\geq 0$ for all $a\leq t,s\leq b;$

\vspace{2mm}

(ii) $\max\limits_{a\leq t\leq b} G(t,s) = G(s,s),\,s\in[a,b];$

\vspace{2mm}

(iii) $G(s,s)$ has a unique maximum, given by

\begin{equation}\label{14}\max_{a\leq s\leq b} G(s,s) =\frac{(\alpha-1)^{\alpha-1}}{(\gamma+\alpha-2)^{\gamma+\alpha-2}} \frac{\left((\gamma-1)b-(\alpha-1)a\right)^{\gamma-1}}{\Gamma(\alpha)(b-a)^{\gamma-\alpha}}.
\end{equation}
\end{lemma}
\begin{proof}We start by defining two functions $$G_+(t,s)=\left(\frac{t-a}{b-a}\right)^{\gamma-1}\left({b - s} \right)^{\alpha  - 1}-\left({t - s} \right)^{\alpha  - 1},\,\,a\leq s\leq t\leq b$$ and $$G_-(t,s)=\left(\frac{t-a}{b-a}\right)^{\gamma-1}\left({b - s} \right)^{\alpha  - 1},\,\,a\leq t\leq s\leq b.$$ It is clear that $$G_-(t,s)\geq 0,\quad \textrm{for all}\quad a\leq t\leq s\leq b.$$ Now, regarding the function $G_+(t,s),$ we have that
\begin{align*}G_+(t,s)&=\left(\frac{t-a}{b-a}\right)^{\gamma-1}\left({b - s} \right)^{\alpha  - 1}-\left({t - s} \right)^{\alpha  - 1}\\&=\left(\frac{t-a}{b-a}\right)^{\gamma-1}\left({b - s} \right)^{\alpha  - 1}\\&-\left(\frac{t-a}{b-a}\right)^{\alpha-1} \left(b-\left(a+\frac{(s-a)(b - a)}{t-a} \right) \right)^{\alpha  - 1}.
\end{align*} Observe now that $$a+\frac{(s-a)(b - a)}{t-a}\geq s,$$ then
\begin{align*}G_+(t,s)&\geq\left({b - s} \right)^{\alpha  - 1}\left[\left(\frac{t-a}{b-a}\right)^{\gamma-1}-\left(\frac{t-a}{b-a}\right)^{\alpha-1}\right],\end{align*} and therefore $G_+(t,s)\geq 0,$ which concludes the proof of (i).

The function $G_-(t,s)$ is an increasing function in $t$ and a decreasing function in $s.$ Then $$0\leq G_-(t,s)\leq G_-(s,s).$$ Using $$\left(\frac{t-a}{b-a}\right)^{\gamma-1}\leq 1,\quad \textrm{for}\quad a\leq t\leq s\leq b,$$ then
\begin{align*}\frac{\partial G_+}{\partial s}(t,s)&=(\alpha-1)\left((t-s)^{\alpha-2}- \left(\frac{t-a}{b-a}\right)^{\gamma-1}(b-s)^{\alpha-2}\right)\\& \geq(\alpha-1)\left((t-s)^{\alpha-2}-(b-s)^{\alpha-2}\right)\geq 0.
\end{align*} Hence, for a given $t,$ the function $G_+(t,s)$ is an increasing function of $s\in [a,t].$ Consequently, $$G_+(t,s)\leq G_+(t,t),$$ which concludes the proof of (ii).

Finally, let $$f(s)=G(s,s)=\left(\frac{s-a}{b-a}\right)^{\gamma-1}\left({b - s} \right)^{\alpha  - 1},\,s\in[a,b].$$ Now, one can verify that \begin{align*}f'(s)&=G'(s,s)\\&=(\gamma-1)(b-a)\left(\frac{s-a}{b-a}\right)^{\gamma-2}\left({b - s} \right)^{\alpha  - 1}\\&-(\alpha-1)\left(\frac{s-a}{b-a}\right)^{\gamma-1}\left({b - s} \right)^{\alpha  - 2}.\end{align*} Observe that $f'(s)$ has a unique zero, attained at the point
$$s=s^*=\frac{(\gamma-1)b+(\alpha-1)a}{\gamma+\alpha-2}.$$ Since, $f''(s^*)\leq 0$ (it is easy to check, for example, in Maple), we conclude that
\begin{align*}\max\limits_{s\in[a,b]}f(s)& =f(s^*)\\&=\frac{(\alpha-1)^{\alpha-1}}{(\gamma+\alpha-2)^{\gamma+\alpha-2}} \frac{\left((\gamma-1)b-(\alpha-1)a\right)^{\gamma-1}}{(b-a)^{\gamma-\alpha}}.\end{align*} This gives $$\max\limits_{s\in[a,b]}G(s,s)=\frac{(\alpha-1)^{\alpha-1}}{(\gamma+\alpha-2)^{\gamma+\alpha-2}} \frac{\left((\gamma-1)b-(\alpha-1)a\right)^{\gamma-1}}{\Gamma(\alpha)(b-a)^{\gamma-\alpha}}.$$ This completes the proof of the Lemma.
\end{proof}

\begin{theorem}\label{th1} If the fractional boundary value problem \begin{equation}\label{8} \left\{\begin{array}{l}D^{\alpha, \mu}_{a}u(t)+q(t)u(t)=0, \,a<t<b,\\{}\\u(a)=u(b)=0,\end{array}\right.\end{equation} has a nontrivial solution, where q is a real and continuous function, then \begin{equation}\label{9} \int\limits_a^b|q(s)|ds> \frac{(\gamma+\alpha-2)^{\gamma+\alpha-2}}{(\alpha-1)^{\alpha-1}} \frac{\Gamma(\alpha)(b-a)^{\gamma-\alpha}}{\left((\gamma-1)b-(\alpha-1)a\right)^{\gamma-1}}.\end{equation}
\end{theorem}
\begin{proof} We equip $C\left([a,b]\right)$ with the Chebyshev norm $$\|u\|=\sup\limits_{t\in [a,b]}|u|.$$ It follows from Lemma \ref{lm1} that a solution to the fractional boundary value problem \eqref{8} satisfies the integral equation \eqref{*}.
Hence, $$\|u\|\leq\max\limits_{t\in [a,b]}\int\limits_a^b \|u\||G(t,s)q(s)|ds,$$ or, equivalently, $$1\leq\max\limits_{t\in [a,b]}\int\limits_a^b |G(t,s)q(s)|ds.$$ Using now the properties of the Green function $G$ proved in Lemma \ref{lm2}, we get \begin{align*}1\leq\max\limits_{t\in [a,b]}\int\limits_a^b |G(t,s)q(s)|ds\\& \leq \max\limits_{t\in [a,b]}\int\limits_a^b |G(t,s)| |q(s)|ds\\&\leq \max\limits_{s\in [a,b]}\int\limits_a^b |G(s,s)| |q(s)|ds \\&\leq\frac{(\alpha-1)^{\alpha-1}}{(\gamma+\alpha-2)^{\gamma+\alpha-2}} \frac{\left((\gamma-1)b-(\alpha-1)a\right)^{\gamma-1}}{\Gamma(\alpha)(b-a)^{\gamma-\alpha}}\int\limits_a^b |q(s)|ds.\end{align*}
Hence
\begin{align*}1\leq\frac{(\alpha-1)^{\alpha-1}}{(\gamma+\alpha-2)^{\gamma+\alpha-2}} \frac{\left((\gamma-1)b-(\alpha-1)a\right)^{\gamma-1}}{\Gamma(\alpha)(b-a)^{\gamma-\alpha}}\int\limits_a^b |q(s)|ds,\end{align*} from which the inequality \eqref{9} follows.
\end{proof}
\begin{remark}Note that if we set $\alpha = 2$ and $\gamma=2$ in \eqref{9}, we obtain Lyapunov's classical inequality \eqref{2}.\end{remark}

\begin{remark}Note that if we set $\gamma=\alpha$ in \eqref{9}, we obtain Lyapunov-type inequality \eqref{4}.\end{remark}

We will end this work by presenting an application of Theorem \ref{th1}. More specifically, we will show how inequality \eqref{9} can be used to determine intervals for the real zeros of the Mittag-Leffler function: $$E_{\alpha, \delta}(z)=\sum\limits_{k=0}^{\infty}\frac{z^k}{\Gamma(\alpha k+\delta)},\,\,\delta, z\in \mathbb{C}\,\Re(\alpha)>0.$$

Let now $a = 0$ and $b = 1$ for simplicity and consider the following fractional Sturm--Liouville type eigenvalue problem: \begin{equation}\label{10} \left\{\begin{array}{l}D^{\alpha, \mu}_{a}u(t)+\lambda u(t)=0, \,0<t<1,\\{}\\u(0)=u(1)=0.\end{array}\right.\end{equation}

We also mention that analogous fractional Sturm--Liouville problem have been studied in \cite{TokmagambetovTorebek, KlimekAgrawal, RiveroTrujilloVelasko}.

\begin{corollary}Let $\lambda\in \mathbb{R}$ be an eigenvalue of problem \eqref{10}. Then $$|\lambda|> \frac{(\gamma+\alpha-2)^{\gamma+\alpha-2}}{(\alpha-1)^{\alpha-1}} \frac{\Gamma(\alpha)}{\left(\gamma-1\right)^{\gamma-1}}.$$\end{corollary}
\begin{corollary}If $$|\lambda|\leq \frac{(\gamma+\alpha-2)^{\gamma+\alpha-2}}{(\alpha-1)^{\alpha-1}} \frac{\Gamma(\alpha)}{\left(\gamma-1\right)^{\gamma-1}},$$ then the system of eigenfunctions $$u(\lambda, t)=t^{\gamma-1}E_{\alpha, \gamma}\left(-\lambda t^{\alpha}\right)$$ of eigenvalue problem \eqref{10} has no real zeros.\end{corollary}

\begin{remark}Note that if we set $\alpha = 2,$ $\gamma=2$ in \eqref{10}, we obtain the classical Sturm-Liouville eigenvalue problem $$u''(t)+\lambda u(t)=0,\,t\in(0,1),$$ $$u(0)=u(1)=0.$$ Then $\lambda=(\pi k)^2,\,k\in \mathbb{N}.$ Hence it is obvious that Sturm-Liouville problem has no nontrivial eigenfunctions if $\lambda\in[-2,2].$\end{remark}

\begin{corollary}If $$|\lambda|\leq \frac{(\gamma+\alpha-2)^{\gamma+\alpha-2}}{(\alpha-1)^{\alpha-1}} \frac{\Gamma(\alpha)}{\left(\gamma-1\right)^{\gamma-1}},$$ then the problem \eqref{10} has no nontrivial solutions in the class of real functions.\end{corollary}

\begin{theorem}\label{th2}The eigenvalue problem \eqref{10} has an infinite number of eigenvalues, and they are the roots of the Mittag-Leffler function $E_{\alpha, \gamma}\left(-\lambda \right),$ i.e. the eigenvalues satisfy \begin{equation}\label{12}E_{\alpha, \gamma}\left(-\lambda \right)=0.\end{equation}
\end{theorem}
\begin{proof}Using the results in \cite{ShinaliyevTurmetovUmarov} and by property \ref{pr4}, a solution of fractional differential equation in \eqref{10} reads \begin{multline}\label{11}u(t)=c_1 t^{\gamma-2}E_{\alpha, \gamma-1}\left(-\lambda t^{\alpha}\right)+c_2 t^{\gamma-1}E_{\alpha, \gamma}\left(-\lambda t^{\alpha}\right)\\=c_1\sum\limits_{k=0}^{\infty}(-\lambda)^k\frac{t^{\alpha k+\gamma-2}}{\Gamma(\alpha k+\gamma-1)}+c_2\sum\limits_{k=0}^{\infty}(-\lambda)^k\frac{t^{\alpha k+\gamma-1}}{\Gamma(\alpha k+\gamma)},
\end{multline} where $c_1$ and $c_2$ are the real constants.\\
From $u(0)=0$ we obtain $c_1=0.$ Since $u(1)=0,$ we get $c_2 E_{\alpha, \gamma}\left(-\lambda \right)=0.$

According to well-known results (see. for example \cite{Pskhu, PopovSedletskii}) Mittag-Leffler type function $E_{\alpha, \gamma}\left(-\lambda \right)$ can has real roots. For more accurate results we give the picture of function $E_{\alpha, \gamma}\left(-\lambda \right)$ (at $\gamma=2$) for particular values of $3/2\leq\alpha< 2$ and $\lambda$ (see. Figure \ref{fig1}):

\begin{figure}[h!]
  \includegraphics[width=0.44\linewidth]{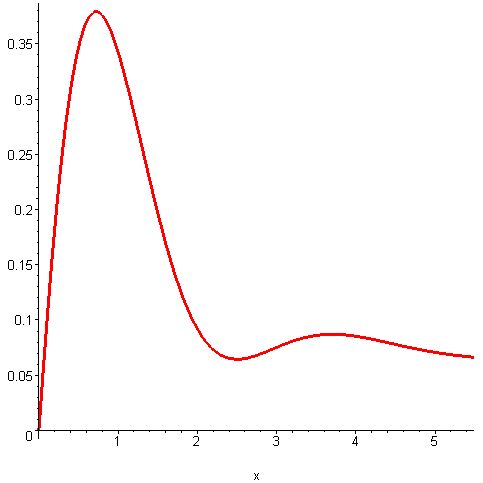}
  \includegraphics[width=0.44\linewidth]{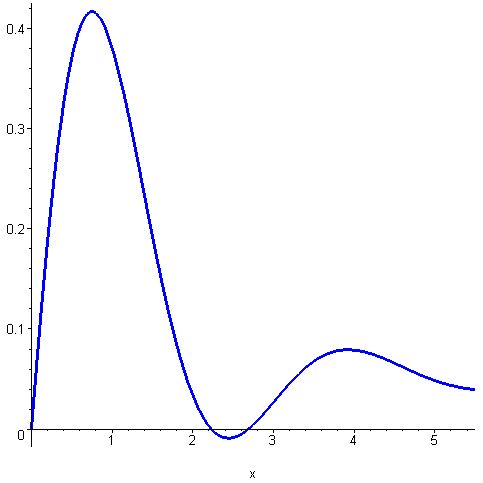}
  \includegraphics[width=0.44\linewidth]{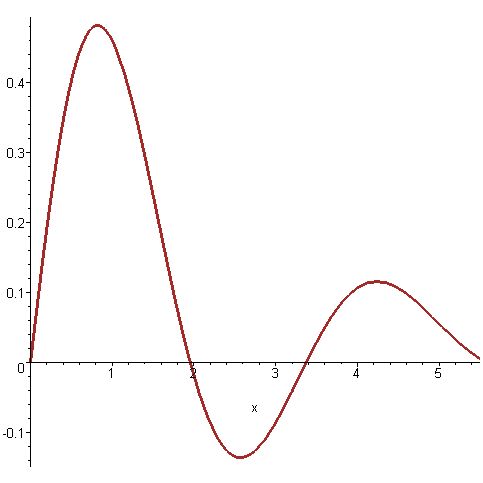}
  \includegraphics[width=0.44\linewidth]{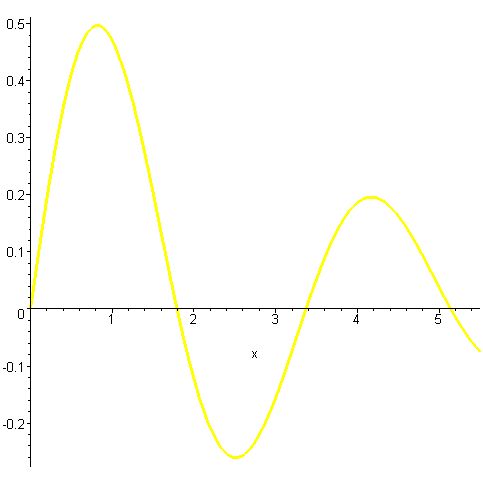}
  \ \ \caption{}
  \label{fig1}
\end{figure}
Eigenfunctions of the problem \eqref{10} has the form $$u(\lambda, t)=t^{\gamma-1}E_{\alpha, \gamma}\left(-\lambda t^{\alpha}\right),$$ where $\lambda$ are the roots of the Mittag-Leffler function $E_{\alpha, \gamma}\left(-\lambda \right).$
\end{proof}

\section{Existence of positive solutions of problem \eqref{7}}\label{Sec4}

\begin{definition} Let $X$ be a real Banach space. A nonempty closed convex set $K\subset X$ is called a cone if it satisfies the following two conditions:

\vspace{2mm}

(i) $x\in K,$ $\lambda\geq 0,$ implies $\lambda x\in K;$

\vspace{2mm}

(ii) $x\in K,$ $-x\in K,$ implies $x=0.$
\end{definition}

\begin{lemma}\label{lmGK}\cite{Krasnoselskii} (Krasnoselskii fixed point theorem). Let $X$ be a Banach space and let $K\subset X$ be a cone. Assume $\Omega_1$ and $\Omega_2$ are bounded open subsets of $X$ with $0\in \Omega_1\subset \overline{\Omega}_1\subset \Omega_2,$ and let $$T:\,K\cap \left(\overline{\Omega}_2\backslash \Omega_1\right) \rightarrow K$$ be a completely continuous operator such that

\vspace{2mm}

(i) $\|Tu\|\geq \|u\|$ for any $u\in K\cap \partial\Omega_1$ and $\|Tu\|\leq \|u\|$ for any $u\in K\cap \partial\Omega_2;$\\or

\vspace{2mm}

(ii) $\|Tu\|\leq \|u\|$ for any $u\in K\cap \partial\Omega_1$ and $\|Tu\|\geq \|u\|$ for any $u\in K\cap \partial\Omega_2.$

\vspace{2mm}

Then, the operator $T$ has a fixed point in $K\cap \left(\overline{\Omega}_2\backslash \Omega_1\right).$
\end{lemma}

\begin{lemma}\label{lm3}For the Green function $G(t,s)$ defined by Lemma \ref{lm1}, a positive function $\varphi\in C(a,b)$ exists such that $$\min\limits_{t\in\left[\frac{3a+b}{4}, \frac{3b+a}{4}\right]}G(t,s)\geq \varphi(s)G(s,s),\, a<s<b.$$
\end{lemma}

\begin{proof}From Lemma \ref{lm2} known that $$G(t,s)\geq 0,\,t,s\in[a,b].$$ Moreover, we know that Green's function $G(t,s)$ is decreasing with respect to $t$ for $s\leq t$ and increasing with respect to $t$ for $t\leq s.$ Obviously, $$G(t,s)> 0,\,t,s\in(a,b)$$ and one can seek the minimum in the interval $\left[\frac{3a+b}{4}, \frac{3b+a}{4}\right].$ Then for $t\in \left[\frac{3a+b}{4}, \frac{3b+a}{4}\right]$ we have
\begin{align*}\min\limits_{t\in\left[\frac{3a+b}{4}, \frac{3b+a}{4}\right]}G(t,s)\\&= \left\{\begin{array}{l}G_+\left(\frac{3b+a}{4},s\right),\qquad\qquad\qquad\qquad \,\,\,\,\,\,\,\,\,\,\,if\,\,s\in\left(a,\frac{3a+b}{4}\right],\\{}\\ \min\left\{G_+\left(\frac{3b+a}{4},s\right),G_+\left(\frac{3a+b}{4},s\right) \right\},\,\,\,\,if\,\,s\in\left[\frac{3a+b}{4}, \frac{3b+a}{4}\right],\\{}\\G_-\left(\frac{3a+b}{4},s\right),\qquad\qquad\qquad\qquad \,\,\,\,\,\,\,\,\,\,\,if\,\,s\in \left[\frac{3b+a}{4},b\right),
\end{array}\right.\\&=\left\{\begin{array}{l}G_+\left(\frac{3b+a}{4},s\right),\,\,\, if\,\,s\in\left(a,r\right],\\{}\\ G_-\left(\frac{3a+b}{4},s\right),\,\,\, if\,\,s\in \left[r,b\right),
\end{array}\right.\\&=\frac{1}{\Gamma(\alpha)} \left\{\begin{array}{l}\left(\frac{3}{4}\right)^{\gamma-1}(b-s)^{\alpha-1}- \frac{(3b+a-4s)^{\alpha-1}}{4^{\alpha-1}}, \,\,\, if\,\,s\in\left(a,r\right],\\{}\\ \frac{1}{4}^{\gamma-1}(b-s)^{\alpha-1},\,\,\, if\,\,s\in \left[r,b\right),
\end{array}\right.
\end{align*} where $\frac{3a+b}{4}<r<\frac{3b+a}{4}$ is the unique solution of equation $$G_+\left(\frac{3b+a}{4},s\right)=G_-\left(\frac{3a+b}{4},s\right).$$ Set $$\varphi(s)= \left\{\begin{array}{l}\frac{G_+\left(\frac{3b+a}{4},s\right)}{G_+\left(s,s\right)}, \,\,\, if\,\,s\in\left(a,r\right],\\{}\\ \frac{G_-\left(\frac{3a+b}{4},s\right)} {G_-\left(s,s\right)},\,\,\, if\,\,s\in \left[r,b\right).
\end{array}\right.$$ The proof is complete. \end{proof}

Define the cone $K\subset C\left([a,b]\right)$ by $$K=\left\{u\in C\left([a,b]\right): u(t) \geq 0,\,a\leq t\leq b \right\}.$$

\begin{lemma}\label{lm4}Let $T: K\rightarrow C\left([a,b]\right)$ be the operator defined by \begin{equation}\label{15}Tu(t)=\int\limits_a^b G(t,s) q(s)f(u(s))ds,\end{equation} then $T: K\rightarrow K$ is completely continuous.
\end{lemma}
\begin{proof}The operator $T: K\rightarrow K$ is continuous in view of non-negativeness and continuity of $G(t,s),$ $q(s)$ and $f(u).$

Let $\Omega\subset K$ be bounded, i.e., there exists a positive constant $M>0$ such that $\|u\|\leq M,$ for all $u\in \Omega.$ Let $$M_0=\max\limits_{a\leq t\leq b,0\leq u\leq M}\left(|qf(u)|+1\right),$$ then, for $u\in\Omega,$ we have $$|Tu(t)|\leq\int\limits_a^b |G(t,s)| |q(s)||f(u(s))|ds\leq M_0\int\limits_a^b G(s,s)ds.$$ Hence, $T$ is bounded in $\Omega.$

On the other hand, given $\varepsilon > 0,$ setting $$\delta=\left(\frac{\Gamma(\alpha)\varepsilon}{M_0(b-a)^{\alpha+1-\gamma}}\right)^{\frac{1}{\gamma-1}}+a,$$ then, for each $u\in\Omega,$ $t_1,t_2\in[a,b],\, t_2<t_1$ and $t_1-t_2<\delta-a,$ one has $$\left|Tu(t_1)-Tu(t_2)\right|<\varepsilon.$$ That is to say, $T$ is equicontinuous in $\Omega.$

In fact, \begin{align*}\left|Tu(t_1)-Tu(t_2)\right|\\&= \left|\int\limits_a^b G(t_1,s) q(s)f(u(s))ds-\int\limits_a^b G(t_2,s) |q(s)f(u(s))|ds\right|\\&= \int\limits_a^{t_1} \left[G(t_1,s) -G(t_2,s)\right] |q(s)f(u(s))|ds\\& = \int\limits_{t_1}^{t_2} \left[G(t_1,s) -G(t_2,s)\right] |q(s)f(u(s))|ds\\&= \int\limits_{t_2}^b \left[G(t_1,s) -G(t_2,s)\right] |q(s)f(u(s))|ds\\& <\frac{M_0}{\Gamma(\alpha)(b-a)^{\gamma-1}}\int\limits_a^{t_1}(b-s)^{\alpha-1} \left[(t_1-a)^{\gamma-1}-(t_2-a)^{\gamma-1}\right]ds\\&+ \frac{M_0}{\Gamma(\alpha)(b-a)^{\gamma-1}}\int\limits_{t_1}^{t_2}(b-s)^{\alpha-1} \left[(t_1-a)^{\gamma-1}-(t_2-a)^{\gamma-1}\right]ds\\&+ \frac{M_0}{\Gamma(\alpha)(b-a)^{\gamma-1}}\int\limits_{t_2}^b(b-s)^{\alpha-1} \left[(t_1-a)^{\gamma-1}-(t_2-a)^{\gamma-1}\right]ds\\& <\frac{M_0}{\Gamma(\alpha)(b-a)^{\gamma-\alpha-1}} \left[(t_1-a)^{\gamma-1}-(t_2-a)^{\gamma-1}\right].\end{align*}
In the following, we divide the proof into two cases.

{\it Case 1.} Let $\delta\leq t_2<t_1<b,$ then by using the mean value theorem we have \begin{align*}\left|Tu(t_1)-Tu(t_2)\right|\\& <\frac{M_0}{\Gamma(\alpha)(b-a)^{\gamma-\alpha-1}} \left[(t_1-a)^{\gamma-1}-(t_2-a)^{\gamma-1}\right]\\&\leq \frac{M_0(\gamma-1)(\delta-a)^{\gamma-2}}{\Gamma(\alpha)(b-a)^{\gamma-\alpha-1}} \left[t_1-t_2\right]\\&\leq \frac{M_0(\delta-a)^{\gamma-2}}{\Gamma(\alpha)(b-a)^{\gamma-\alpha-1}} \left[t_1-t_2\right]\leq \varepsilon.\end{align*}

{\it Case 2.} Let $a\leq t_2<t_1<\delta.$ Then \begin{align*}\left|Tu(t_1)-Tu(t_2)\right|\\& <\frac{M_0}{\Gamma(\alpha)(b-a)^{\gamma-\alpha-1}} \left[(t_1-a)^{\gamma-1}-(t_2-a)^{\gamma-1}\right]\\&\leq \frac{M_0}{\Gamma(\alpha)(b-a)^{\gamma-\alpha-1}}(\delta-a)^{\gamma-1}\leq \varepsilon.\end{align*} By the Arzela--Ascoli theorem, the operator $T: K\rightarrow K$ is completely continuous.
The proof is complete.
\end{proof}

Denote \begin{equation}\label{theta}\theta=\left(\int\limits^b_a G(s,s) q(s)ds\right)^{-1}\end{equation} and \begin{equation}\label{theta*}\theta^*=\left(\int\limits^{\frac{3b+a}{4}}_{\frac{3a+b}{4}} G(s,s) \varphi(s)q(s)ds\right)^{-1}.\end{equation}

To prove the existence of nontrivial positive solutions to the fractional boundary
value problem \eqref{7} we consider the following hypotheses:

\vspace{2mm}

(A) $f(u)\geq \theta^*r_1$ for $u\in[0,r_1];$

\vspace{2mm}

(B) $f(u)\leq \theta r_2$ for $u\in[0,r_2],$
\vspace{2mm}\\
where $f:\,\mathbb{R}_+\rightarrow \mathbb{R}_+$ is continuous.

\begin{theorem}\label{th3}Let $q:\,[a,b]\rightarrow \mathbb{R}_+$ be a nontrivial Lebesgue integrable function. Assume that there exists two positive constants $0<r_1<r_2$ such that the assumptions (A) and (B) are satisfied. Then the fractional boundary value problem \eqref{7} has at least one nontrivial positive solution $u$ belonging to $X$ such that $$r_1\leq \|u\|\leq r_2.$$
\end{theorem}
\begin{proof} Using the Arzela-Ascoli theorem, we prove that $T: K\rightarrow K$ is completely continuous operator (see. lemma \ref{lm4}). Let $$\Omega_i=\left\{u\in K:\, \|u\|\leq r_i\right\}.$$ From (A) and Lemma \ref{lm3}, we have for $t\in\left[\frac{3a+b}{4},\frac{3b+a}{4}\right]$ and $u\in K\cap \partial\Omega_1$ that \begin{align*}Tu(t)&=\int\limits^b_a G(t,s) q(s) f(u(s))ds\\&\geq \int\limits^b_a \min\limits_{t\in\left[\frac{3a+b}{4},\frac{3b+a}{4}\right]}G(t,s) q(s)f(u(s))ds\\&\geq \int\limits^b_a \varphi(s)G(s,s) q(s)f(u(s))ds\\&\geq \theta^* r_1\int\limits^b_a \varphi(s)G(s,s) q(s)ds\\& \geq \theta^* r_1\int\limits^\frac{3b+a}{4}_\frac{3a+b}{4} \varphi(s)G(s,s) q(s)ds\geq \|u\|.\end{align*}
Thus, $\|Tu\|\geq \|u\|$ for any $u\in K\cap \partial\Omega_1.$  Let us now prove that $\|Tu\|\leq \|u\|$ for all $u\in K\cap \partial\Omega_2.$  From (B), it follows that \begin{align*}\|Tu(t)\|&=\left|\int\limits^b_a \max\limits_{t\in[a,b]}G(t,s) q(s) f(u(s))ds\right|\\&\leq \int\limits^b_a \max\limits_{t\in[a,b]}G(t,s) q(s) f(u(s))ds \\&\leq\int\limits^b_a G(s,s) q(s) f(u(s))ds\\& \leq\theta r_2\int\limits^b_a G(s,s) q(s)ds\leq \|u\|, \end{align*} for $u\in K\cap \partial\Omega_2.$ Thus, from Krasnoselskii fixed point theorem (see. lemma \ref{lmGK}) we conclude that the operator $T$ defined by \eqref{15} has a fixed point in $K\cap \left(\overline{\Omega}_2\backslash \Omega_1\right).$ Therefore, the fractional boundary problem \eqref{7} has at least one positive solution $u$ belonging to $X$ such that $$r_1\leq u\leq r_2.$$ The proof is complete.
\end{proof}

\begin{example}Let $\alpha:=\frac{7}{4},$ $\gamma:=2,$ $a=0,\,b=1,$ $q(t)=t^2,$ $f(u)=\cosh u$ in \eqref{7}. Then we have the following fractional boundary value problem: \begin{equation}\label{ex1}\left\{\begin{array}{l}D_0^{\frac{7}{4},2}u(t)+t^2\cosh u(t)=0,\,t\in(0,1),\\{}\\u(0)=u(1)=0.\end{array}\right.\end{equation} Firstly, let us calculate the values of $\theta$ and $\theta^*.$ Here $$\varphi(s)=\left\{\begin{array}{l}\frac{3}{4s}-\frac{1}{2\sqrt{2}s} \left(\frac{3-4s}{1-s}\right)^{\frac{3}{4}},\,s\in(0,r],\\{}\\ \frac{1}{4s},\,s\in[r,1) \end{array}\right.$$ where $r\approx 0,58.$ Hence, by a simple computation, we get $$\theta\approx 8,9\quad \textrm{and}\quad \theta^*\approx 11,61.$$ Choosing $r_1=\frac{1}{12},\, r_2=\frac{1}{8}$ we obtain $$f(u)=\cosh u\geq \theta^*r_1,\quad \textrm{for}\quad u\in\left[0,\frac{1}{12}\right].$$
$$f(u)=\cosh u\leq \theta r_2,\quad \textrm{for}\quad u\in\left[0,\frac{1}{8}\right].$$ Therefore, from Theorem \ref{th3}, problem \eqref{ex1} has at least one nontrivial solution $u$ in $C\left([0,1]\right)$ such that $\frac{1}{12}\leq\|u\|\leq\frac{1}{8}.$

\end{example}

\section{Lyapunov-type inequality for the problem \eqref{7}}\label{Sec5}

The next result generalizes Theorem \ref{th1} by choosing $f(u) = u$ in Theorem \ref{th4},
inequality \eqref{16} reduces to \eqref{9}. Note that $f \in C(\mathbb{R}_+, \mathbb{R}_+)$ is a concave and nondecreasing function.

\begin{lemma}\label{lmJ} \cite{Rubin} (Jensen's inequality). Let $\mu$ be a positive measure and let $\Omega$ be a measurable set with $\mu(\Omega)=1.$ Let $I$ be an interval and suppose that $u$ is a real function in $L(d\mu)$ with $u(t)\in I$ for all $t\in\Omega.$ If $f$ is convex on $I,$ then \begin{equation}\label{13}f\left(\int\limits_\Omega u(t)d\mu(t) \right)\leq \int\limits_\Omega \left(f\circ u\right)(t)d\mu(t).\end{equation} If $f$ is concave on $I,$ then the inequality \eqref{13} holds with "$\leq$" substituted by "$\geq$".\end{lemma}

Let $$\|u\|_{L^1}=\|u\|_{L^1\left([a,b]\right)}=\int\limits_a^b|u(s)|ds,$$ where $L^1\left([a,b]\right)$ the space of all real functions, defined on $[a,b],$ which are Lebesgue integrable.

\begin{theorem}\label{th4}Let $q:\,[a,b]\rightarrow \mathbb{R}$ be a real nontrivial Lebesgue integrable function. Assume that $f\in C\left(\mathbb{R}_+, \mathbb{R}_+\right)$ is a concave and nondecreasing function. If the fractional boundary value problem \eqref{7} has a nontrivial solution $u$, then \begin{equation}\label{16} \int\limits_a^b|q(s)|ds> \frac{(\gamma+\alpha-2)^{\gamma+\alpha-2}}{(\alpha-1)^{\alpha-1}} \frac{\Gamma(\alpha)(b-a)^{\gamma-\alpha}} {\left((\gamma-1)b-(\alpha-1)a\right)^{\gamma-1}}\frac{\omega}{f(\omega)},\end{equation} where $\omega=\max\limits_{t\in[a,b]}u(t).$
\end{theorem}

\begin{proof}The proof makes use of Lemma \ref{lm1}. We have $$|u(t)|\leq\int\limits_a^b G(t,s) |q(s)||f(u(s))|ds.$$ So \begin{align*}\|u(t)\|&\leq\int\limits_a^b G(t,s) |q(s)||f(u(s))|ds\\& \leq\int\limits_a^b \max\limits_{s\in[a,b]}G(s,s) |q(s)||f(u(s))|ds\\& \leq\frac{(\alpha-1)^{\alpha-1}}{(\gamma+\alpha-2)^{\gamma+\alpha-2}} \frac{\left((\gamma-1)b-(\alpha-1)a\right)^{\gamma-1}} {\Gamma(\alpha)(b-a)^{\gamma-\alpha}}\int\limits_a^b |q(s)||f(u(s))|ds.\end{align*} Using Jensen's inequality \eqref{13}, and taking into account that $f$ is concave and nondecreasing, we get that \begin{align*}\|u(t)\|& \leq\frac{(\alpha-1)^{\alpha-1}}{(\gamma+\alpha-2)^{\gamma+\alpha-2}} \frac{\left((\gamma-1)b-(\alpha-1)a\right)^{\gamma-1}} {\Gamma(\alpha)(b-a)^{\gamma-\alpha}}\int\limits_a^b |q(s)||f(u(s))|ds\\&  =\frac{(\alpha-1)^{\alpha-1}}{(\gamma+\alpha-2)^{\gamma+\alpha-2}} \frac{\left((\gamma-1)b-(\alpha-1)a\right)^{\gamma-1}} {\Gamma(\alpha)(b-a)^{\gamma-\alpha}}\|q\|_{L^1}\int\limits_a^b \frac{|q(s)||f(u(s))|}{\|q\|_{L^1}}ds \\&\leq \frac{(\alpha-1)^{\alpha-1}}{(\gamma+\alpha-2)^{\gamma+\alpha-2}} \frac{\left((\gamma-1)b-(\alpha-1)a\right)^{\gamma-1}} {\Gamma(\alpha)(b-a)^{\gamma-\alpha}}\|q\|_{L^1}f\left(\int\limits_a^b \frac{|q(s)||u(s)|}{\|q\|_{L^1}}ds\right)\\ &\leq \frac{(\alpha-1)^{\alpha-1}}{(\gamma+\alpha-2)^{\gamma+\alpha-2}} \frac{\left((\gamma-1)b-(\alpha-1)a\right)^{\gamma-1}} {\Gamma(\alpha)(b-a)^{\gamma-\alpha}}f\left(\|u\|\right)\int\limits_a^b |q(s)|ds. \end{align*}

Thus, $$\|u(t)\|\leq \frac{(\alpha-1)^{\alpha-1}}{(\gamma+\alpha-2)^{\gamma+\alpha-2}} \frac{\left((\gamma-1)b-(\alpha-1)a\right)^{\gamma-1}} {\Gamma(\alpha)(b-a)^{\gamma-\alpha}}f\left(\|u\|\right)\int\limits_a^b |q(s)|ds.$$ This concludes the proof.
\end{proof}

\begin{corollary}\label{cor1} Consider the fractional boundary value problem \eqref{7} with $f\in C\left(\mathbb{R}_+, \mathbb{R}_+\right)$  concave and nondecreasing and $q\in L^1\left([a,b], \mathbb{R}_+\right).$ If there exists two positive constants $0<r_1<r_2$ such that $f(u)\geq \theta^*r_1$ for $u\in[0,r_1]$ and $f(u)\leq \theta r_2$ for $u\in[0,r_2],$ then $$\int\limits_a^b q(s)ds> \frac{(\gamma+\alpha-2)^{\gamma+\alpha-2}}{(\alpha-1)^{\alpha-1}} \frac{\Gamma(\alpha)(b-a)^{\gamma-\alpha}} {\left((\gamma-1)b-(\alpha-1)a\right)^{\gamma-1}}\frac{r_1}{f(r_2)}.$$
\end{corollary}

\begin{example} Consider the following fractional boundary value problem: \begin{equation*}\left\{\begin{array}{l} D^{\frac{3}{2},\frac{3}{2}}_0 u(t)+\sqrt{t}\exp\left(-\frac{1}{u+1}\right)=0,\,0<t<1,\\{}\\ u(0)=u(1)=0. \end{array}\right.\end{equation*} We have that\\
$f(u)=\exp\left(-\frac{1}{u+1}\right):\, \mathbb{R}_+\rightarrow \mathbb{R}_+$ is continuous, concave and nondecreasing function;\\
$q(t)=\sqrt(t):\, [0,1]\rightarrow \mathbb{R}_+$ is a Lebesgue integral function and $\|q\|_L^1=\frac{2}{3}>0.$

We computed the values of $\theta$ and $\theta^*$ in \eqref{theta} and \eqref{theta*}. Then we have $\theta\approx 4.23$ and $\theta^*\approx 7.29.$ Choosing $r_1=\frac{1}{20}$ and $r_2=\frac{1}{10},$ we obtain \begin{align*}&f(u)=\exp\left(-\frac{1}{u+1}\right)\geq \theta^*r_1,\quad \textrm{for}\quad u\in\left[0,\frac{1}{20}\right];\\& f(u)=\exp\left(-\frac{1}{u+1}\right)\leq \theta r_2,\quad \textrm{for}\quad u\in\left[0,\frac{1}{10}\right].\end{align*}
Therefore, from Corollary \ref{cor1}, we get that $$\int\limits_a^b q(s)ds=\frac{2}{3}> \frac{\Gamma\left(\frac{3}{2}\right)\exp(10/11)}{10}\approx 0.22.$$

\end{example}

\begin{remark}Note that if we set $\gamma=\alpha$ in \eqref{16}, then we obtain Lyapunov-type inequality \eqref{6**}.\end{remark}

\begin{theorem}\label{th5}Let $f\in C\left(\mathbb{R}_+, \mathbb{R}_+\right)$  be nondecreasing and $q\in L^1\left([a,b], \mathbb{R}_+\right).$ If $$f(u)<\theta |u|,$$ then the problem \eqref{7} has no non-trival solution.\\ Here $\theta=\left(\int\limits^b_a G(s,s) q(s)ds\right)^{-1}.$
\end{theorem}

\begin{proof} Assume, on the contrary, that there exists $u\neq 0.$ Then, by Lemma \ref{lm2} we have \begin{align*}\|u\|&=\left|\int\limits_a^b G(t,s) q(s)f(u(s))ds\right| \leq \max\limits_{t\in[a,b]}\int\limits_a^b G(t,s) |q(s)||f(u(s))|ds\\& \leq\int\limits_a^b G(s,s) |q(s)||f(u(s))|ds < \theta \int\limits_a^b G(s,s) |q(s)||u|ds\\&\leq \theta \|u\|\int\limits_a^b G(s,s) |q(s)|ds=\|u\|,\end{align*} a contradiction. Therefore, problem \eqref{7} has no solution.
\end{proof}

\section{Hartman-Wintner-Type Inequality for the problem \eqref{7}}\label{Sec6}

Our main result in this section is the following Hartman-Wintner-type inequality.

\begin{theorem}\label{th6} Let the functions $q$ and $f$ satisfy the condition of Theorem \ref{th4}. Suppose that fractional boundary value problem \eqref{7} has a nontrivial continuous solution. Then \begin{equation}\label{17} \int\limits_a^b \left(s-a\right)^{\gamma-1}\left({b - s} \right)^{\alpha  - 1}q^+(s) ds > \frac{\|u(t)\|}{f\left(\|u\|\right)} \Gamma(\alpha)(b-a)^{\gamma-1}.\end{equation} \end{theorem}

\begin{proof} By Lemma \ref{lm1}, a solution $u\in C\left([a,b]\right)$ to \eqref{7} has the expression
$$u(t)=\int\limits_a^b G(t,s)q(s)f\left(u(s)\right)ds,\,a\leq t\leq b.$$ From this, for any $a\leq t\leq b,$ we obtain \begin{align*}|u(t)|&\leq\int\limits_a^b \left|G(t,s)q(s)f\left(u(s)\right)\right|ds \leq \int\limits_a^b |G(t,s)|\left|q(s)\right|\left|f\left(u(s)\right)\right|ds. \end{align*} By Jensen's inequality \eqref{13}, and taking into account that $f$ is concave and nondecreasing, we get that
\begin{align*}|u(t)|&\leq \int\limits_a^b G(s,s)\left|q(s)\right|\left|f\left(u(s)\right)\right|ds \\&\leq \left\|G(s,s) q(s)\right\|_{L^1} \int\limits_a^b \frac{G(s,s)\left|q(s)\right|\left|f\left(u(s)\right)\right|}{\left\|G(s,s) q(s)\right\|_{L^1}}ds \\&\leq \left\|G(s,s) q(s)\right\|_{L^1} f\left(\int\limits_a^b \frac{G(s,s)\left|q(s)\right|\left|u(s)\right|}{\left\|G(s,s) q(s)\right\|_{L^1}}ds\right) \\&\leq \left\|G(s,s) q(s)\right\|_{L^1} f\left(\|u\|\int\limits_a^b \frac{G(s,s)\left|q(s)\right|}{\left\|G(s,s) q(s)\right\|_{L^1}}ds\right) \\& \leq \int\limits_a^b G(s,s)\left|q(s)\right|ds f\left(\|u\|\right)\\&\leq \frac{1}{\Gamma(\alpha)(b-a)^{\gamma-1}}\int\limits_a^b \left(s-a\right)^{\gamma-1}\left({b - s} \right)^{\alpha  - 1}q^+(s) ds f\left(\|u\|\right).\end{align*} Then we get \begin{align*}\|u(t)\| \leq \frac{1}{\Gamma(\alpha)(b-a)^{\gamma-1}}\int\limits_a^b \left(s-a\right)^{\gamma-1}\left({b - s} \right)^{\alpha  - 1}q^+(s) ds f\left(\|u\|\right).\end{align*}
The proof is complete.
\end{proof}

\begin{corollary}\label{cor1} If $f(u)=u$ (linear case), for $q\in L^1\left([a,b], \mathbb{R}_+\right)$ we obtain $$\int\limits_a^b \left(s-a\right)^{\gamma-1}\left({b - s} \right)^{\alpha  - 1}q^+(s) ds > \Gamma(\alpha)(b-a)^{\gamma-1}.$$
\end{corollary}

\begin{corollary}\label{cor1} Consider the fractional boundary value problem \eqref{7} with $f\in C\left(\mathbb{R}_+, \mathbb{R}_+\right)$  concave and nondecreasing and $q\in L^1\left([a,b], \mathbb{R}_+\right).$ If there exists two positive constants $0<r_1<r_2$ such that $f(u)\geq \theta^*r_1$ for $u\in[0,r_1]$ and $f(u)\leq \theta r_2$ for $u\in[0,r_2],$ then $$\frac{f\left(r_2\right)}{r_1}\int\limits_a^b \left(s-a\right)^{\gamma-1}\left({b - s} \right)^{\alpha  - 1}q^+(s) ds > \Gamma(\alpha)(b-a)^{\gamma-1}.$$
\end{corollary}

\begin{remark}Note that if we set $\gamma=\alpha=2$ and $f(u)=u$ in \eqref{17}, then we obtain Hartman-Wintner inequality \eqref{HW}.\end{remark}

\begin{remark} As (see \eqref{14}) \begin{multline*}\max\limits_{s\in[a,b]}\left(s-a\right)^{\gamma-1}\left({b - s} \right)^{\alpha  - 1}\\=\frac{(\alpha-1)^{\alpha-1}}{(\gamma+\alpha-2)^{\gamma+\alpha-2}} \frac{\left((\gamma-1)b-(\alpha-1)a\right)^{\gamma-1}}{\Gamma(\alpha)(b-a)^{\gamma-\alpha}},\end{multline*} from Hartman-Wintner-type inequality \eqref{17} we have Lyapunov-type inequality \eqref{16}.\end{remark}

\begin{corollary}Let $\lambda\in \mathbb{R}_+$ be an eigenvalue of the problem \eqref{10}. Then $$\lambda> \frac{\Gamma(\gamma+\alpha)}{\Gamma(\gamma)}(b-a)^{-\alpha}.$$\end{corollary}
\begin{corollary}If $$\lambda \leq \frac{\Gamma(\gamma+\alpha)}{\Gamma(\gamma)}(b-a)^{-\alpha},$$ then the system of eigenfunctions $$u(\lambda, t)=t^{\gamma-1}E_{\alpha, \gamma}\left(-\lambda t^{\alpha}\right)$$ of eigenvalue problem \eqref{10} has no real positive zeros.\end{corollary}

\section*{Acknowledgements}
The second named author is financially supported by a grant from
the Ministry of Science and Education of the Republic of
Kazakhstan (Grant No. 0819/GF4). This publication is supported by the target program 0085/PTSF-14
from the Ministry of Science and Education of the Republic of Kazakhstan.


\begin{thebibliography}{ABGM15}

\bibitem[T10]{Tiryaki10} A. Tiryaki, Recent developments of Lyapunov-type inequalities, Adv. Dyn. Syst. Appl. V. 5, No. 2, 231-–248 (2010).

\bibitem[H15]{Hashizume15} M. Hashizume, Minimization problem related to a Lyapunov inequality, J. Math. Anal. Appl. V. 432, No. 1, 517-–530 (2015).

\bibitem[SL15]{SunLiu15} T. Sun and J. Liu, Lyapunov inequality for dynamic equation with order $n+1$ on time scales, J. Dyn. Syst. Geom. Theor. V. 13, No. 1, 95–-101 (2015).

\bibitem[L93]{Lyapunov93} A.~M.~Liapounoff, Probl\'{e}me g\'{e}n\'{e}ral de la stabilit\'{e} du mouvement, Ann. Fac. Sci. Univ. Toulouse, V. 9, 203--474 (1907).

\bibitem[F13]{Ferreira13} R. A. C. Ferreira, A Lyapunov-type inequality for a fractional boundary value problem, Fract. Calc. Appl. Anal. V. 16, No. 4, 978–-984 (2013).

\bibitem[MMW13]{MaMaWang16} Q. Ma, Ch. Ma and J. Wang, A Lyapunov-type inequality for a fractional differential equation with Hadamard derivative, Journal of Mathematical Inequalities. V. 11, No. 1, 135--141 (2017).

\bibitem[F14]{Ferreira14} R. A. C. Ferreira, On a Lyapunov-type inequality and the zeros of a certain Mittag-Leffler function, J. Math. Anal. Appl. V. 412 No. 2, 1058–1063 (2014).

\bibitem[JS15]{JieliSamet15} M. Jleli and B. Samet, Lyapunov-type inequalities for a fractional differential equation with mixed boundary conditions, Math. Inequal. Appl. V. 18, No. 2, 443-–451 (2015).

\bibitem[O'RS15]{O'ReganSamet15}D. O'Regan and B. Samet, Lyapunov-type inequalities for a class of fractional differential equations, J. Inequal. Appl. V. 2015, No. 247, 1--10 (2015).

\bibitem[RB15]{RongBai15} J. Rong, C. Z. Bai, Lyapunov-type inequality for a fractional differential equations with fractional boundary value problems, Adv. Difference. Equ., V. 2015, No. 82, 1--10 (2015).

\bibitem[EA16]{EshaghiAnsari16} S. Eshaghi, A. Ansari, Lyapunov inequality for fractional differential equations with Prabhakar derivative, Math. Inequal. Appl. V. 19, No. 1, 349--358 (2016).

\bibitem[F16]{Ferreira16} R. A. C. Ferreira, Lyapunov-Type Inequalities for some Sequential Fractional Boundary Value Problems, Advances in Dynamical Systems and Applications, V. 11, No. 1, 33-–43 (2016).

\bibitem[CT17]{ChidouhTorres17} A. Chidouh, D. F. M. Torres, A generalized Lyapunov's inequality for a fractional boundary value problem, Journal of Computational and Applied Mathematics, V. 312, No. 1, 192-–197 (2017).

\bibitem[HW51]{HartmanWintner} P. Hartman, A. Wintner, On an oscillation criterion of Lyapunov, American Journal of Mathematics, V. 73, 885--890 (1951).

\bibitem[CSS17]{CabreraSadaranganoSamet} I. Cabrera, K. Sadarangani, and B. Samet, Hartman-Wintner-type inequalities for a class of nonlocal fractional boundary value problems, Math. Meth. Appl. Sci., V. 40, 129--136 (2017).

\bibitem[JKS17]{JleliKiraneSamet} M. Jleli, M. Kirane, and B. Samet, Hartman-Wintner-Type Inequality for a Fractional Boundary Value Problem via a Fractional Derivative with respect to Another Function, Discrete Dynamics in Nature and Society, V. 2017, doi: 10.1155/2017/5123240 (2017).

\bibitem[BG03]{BabakhaniGejji} A. Babakhani, V. D. Gejji, Existence of positive solutions of nonlinear fractional differential equations, Journal of Mathematical Analysis and Applications, V. 278, 434--442 (2003).

\bibitem[Zh03]{Zhang03} S. Q. Zhang, Existence of positive solution for some class of nonlinear fractional differential equations, Journal of Mathematical Analysis and Applications, V. 278, 136--148 (2003).

\bibitem[K09]{Kaufmann} E. R. Kaufmann, Existence and Nonexistence of positive solutions for a nonlinear fractional boundary value problem, Discrete Contin. Dyn. Syst. suppl. 416--423 (2009).

\bibitem[BRRT07]{Bonillaetal} B. Bonilla, M. Rivero, L. Rodriguez-Germá, J. J. Trujillo, Fractional differential equations as alternative models to nonlinear differential equations, Applied Mathematics and Computations, V. 187, 79--88 (2007).

\bibitem[DB08]{DaftardarBhalekar} V. Daftardar-Gejji, S. Bhalekar, Boundary value problems for multi-term fractional differential equations, Journal of Mathematical Analysis and Applications, V. 345, 754--765 (2008).

\bibitem[CW12]{CabadaWang} A. Cabada, G. Wang, Positive solutions of nonlinear fractional differential equations with integral boundary value conditions, J. Math. Anal. Appl., V. 389, 403--411 (2012).

\bibitem[CI12]{CabadaInfante} A. Cabada, G. Infante, Positive solutions of a nonlocal Caputo fractional BVP, Dynamic Systems and Applications, V. 23, 715--722 (2014).

\bibitem[IR16]{InfanteRihani} G. Infante and S. Rihani, Nontrivial solutions of systems of nonlocal Caputo fractional BVPs, arXiv:1601.08073 (2016).

\bibitem[BL05]{BaiLu} Z. Bai and H. L\"{u}, Positive solutions for boundary value problem of nonlinear fractional differential equation, J. Math. Anal. Appl. V. 311, No. 2, 495--505 (2005).

\bibitem[ShZh09]{ShiZhang} A. Shi, S. Zhang, Upper and lower solutions method and a fractional differential boundary value problem, Electron. J. Qual. Theory Differ. Equ., No. 30, 1--13 (2009).

\bibitem[JY10]{JiangYuan} D. Jiang, C. Yuan, The positive properties of the Green function for Dirichlet-type boundary value problems of nonlinear fractional differential equations and its application, Nonlinear Analysis: Theory, Methods and Applications. V. 72, No. 2, 710--719 (2010).

\bibitem[ZhL08]{ZhangLiu} X. Zhang, L. Liu, Yonghong Wu, Positive solutions of nonresonance semipositone singular Dirichlet boundary value problems,  Nonlinear Analysis: Theory, Methods and Applications, V. 68, No. 1, 97--108 (2008).

\bibitem[KST06]{KilbasSrivastavaTrujillo} A. A. Kilbas, H. M. Srivastava, J. J. Trujillo, Theory and Applications of Fractional Differential Equations. Elsevier. North-Holland. Mathematics studies. 2006. -539p.

\bibitem[H08]{Hilfer} R. Hilfer. Threefold introduction to fractional derivatives, Anomalous transport: Foundations and applications, 17--73 (2008).

\bibitem[FIK14]{FuratiIyiolaKirane} K. M. Furati, O. S. Iyiola and M. Kirane, An inverse problem for a generalized fractional diffusion, Applied Mathematics and Computation, V. 249, 24--31 (2014).

\bibitem[TT16]{TokmagambetovTorebek} N. Tokmagambetov, B. T. Torebek, Fractional Analogue of Sturm-Liouville Operator, Documenta Mathematica, V. 21, 1503--1514 (2016).

\bibitem[KA13]{KlimekAgrawal} M. Klimek, O.P. Agrawal, Fractional Sturm-Liouville problem. Computers and Mathematics with Applications, V. 66, No. 5, 795--812 (2013).

\bibitem[RTV13]{RiveroTrujilloVelasko} M. Rivero, J.J. Trujillo, M.P. Velasco, A fractional approach to the Sturm-Liouville problem. Cent. Eur. J. Phys., V. 11, No 10, 1246--1254 (2013).

\bibitem[STU12]{ShinaliyevTurmetovUmarov} K. M. Shinaliyev, B. Kh. Turmetov, and S. R. Umarov, A Fractional operator algorithm method for construction of solutions of fractional order differential equations, Fractional Calculus and Applied Analysis. V. 15, No. 2, 267--281 (2012).

\bibitem[P05]{Pskhu} A. V. Pskhu, On the real zeros of functions of Mittag-Leffler type, Mathematical Notes, V. 77, No. 4, 546--552 (2005).

\bibitem[PS13]{PopovSedletskii} A. Yu. Popov, A. M. Sedletskii, Distribution of roots of Mittag-Leffler functions, Journal of Mathematical Sciences, V. 190, No. 2, 209--409 (2013).

\bibitem[K64]{Krasnoselskii} M.A. Krasnosel'skii, Positive Solutions of Operator Equations, Noordhoff, Groningen, 1964.

\bibitem[R87]{Rubin} W. Rudin, Real and complex analysis, third edition, McGraw-Hill Book Company, New York, 1987.

\end{thebibliography}
\end{document}